\documentclass[12pt]{article}
\usepackage{amssymb}
\usepackage{amsmath,amsthm}
\usepackage[latin1]{inputenc}
\usepackage{cite}
\usepackage[dvips]{graphicx}
\usepackage{hyperref}
\usepackage{color}
\usepackage{setspace}
\usepackage{enumerate}
\usepackage{tkz-graph}
\usepackage{tikz}

\hypersetup{colorlinks=true}

\hypersetup{colorlinks=true, linkcolor=blue, citecolor=blue,urlcolor=blue}

 \newtheorem{remark}{Remark}

 \newtheorem{lemma}[remark]{Lemma}

 \newtheorem{theorem}[remark]{Theorem}
 \newtheorem{proposition}[remark]{Proposition}
 \newtheorem{corollary}[remark]{Corollary}

\newcommand{\Sd}{\operatorname{Sd}}
\newcommand{\Sgamma}{\operatorname{S\gamma}}

\newcommand{\Perm}{\operatorname{Perm}}

\title{Simultaneous Resolvability in Families of Corona Product Graphs}

\author{Yunior Ram\'{\i}rez-Cruz, Alejandro Estrada-Moreno\\ and Juan A.
Rodr\'{\i}guez-Vel\'{a}zquez
    \\
{\small Departament d'Enginyeria Inform\`atica i Matem\`atiques,}\\
{\small Universitat Rovira i Virgili,}  {\small Av. Pa\"{\i}sos
Catalans 26, 43007 Tarragona, Spain.} \\{\small
yunior.ramirez\@@urv.cat, alejandro.estrada\@@urv.cat, juanalberto.rodriguez\@@urv.cat}
}
\begin{document}
\maketitle

\begin{abstract}
Let ${\cal G}$ be a graph family defined on a common vertex set $V$ and let $d$ be a distance defined on every graph $G\in {\cal G}$. 
A set   $S\subset  V$ is said to be a simultaneous metric generator for ${\cal G}$ if for every $G\in {\cal G}$ and  every pair of different vertices $u,v\in V$ there exists $s\in S$ such that $d(s,u)\ne d(s,v)$.
The  simultaneous metric  dimension of ${\cal G}$ is the smallest integer $k$ such that there is a  simultaneous metric generator for ${\cal G}$  of cardinality
$k$.  We study the  simultaneous metric dimension of families composed by corona product graphs. Specifically, we focus on the case of two  particular distances defined on every $G\in {\cal G}$, namely,  the geodesic distance $d_G$ and  the distance $d_{G,2}:V\times V\rightarrow \mathbb{N}\cup \{0\}$ defined as $d_{G,2}(x,y)=\min\{d_{G}(x,y),2\}$.
\end{abstract}

\section{Introduction}
A generator of a metric space $(X,d)$ is a set $S\subset X$ of points in the space  with the property that every point of $X$  is uniquely determined by the distances from the elements of $S$. Given a simple and connected graph $G=(V,E)$, we consider the function $d_G:V\times V\rightarrow \mathbb{N}\cup \{0\}$, where $d_G(x,y)$ is the length of a shortest path between $u$ and $v$ and $\mathbb{N}$ is the set of positive integers. Then $(V,d_G)$ is a metric space since $d_G$ satisfies $(i)$ $d_G(x,x)=0$  for all $x\in V$,$(ii)$  $d_G(x,y)=d_G(y,x)$  for all $x,y \in V$ and $(iii)$ $d_G(x,y)\le d_G(x,z)+d_G(z,y)$  for all $x,y,z\in V$. A vertex $v\in V$ is said to \textit{distinguish} two vertices $x$ and $y$ if $d_G(v,x)\ne d_G(v,y)$.
A set $S\subset V$ is said to be a \emph{metric generator} for $G$ if any pair of vertices of $G$ is
distinguished by some element of $S$. A minimum cardinality metric generator is called a \emph{metric basis}, and
its cardinality the \emph{metric dimension} of $G$, denoted by $\dim(G)$.  

The notion of metric dimension of a graph was introduced by Slater in \cite{Slater1975}, where metric generators were called \emph{locating sets}. Harary and Melter independently introduced the same concept in  \cite{Harary1976}, where metric generators were called \emph{resolving sets}. 

The concept of adjacency generator\footnote{Adjacency generators were called adjacency resolving sets in   \cite{JanOmo2012}.} was introduced by Jannesari and Omoomi in \cite{JanOmo2012} as a tool to study the metric dimension of lexicographic product graphs. A set $S\subset V$ of vertices in a graph $G=(V,E)$ is said to be  an \emph{adjacency generator} for $G$  if for every two vertices $x,y\in V-S$ there exists $s\in S$ such that $s$ is adjacent to exactly one of $x$ and $y$. A minimum cardinality adjacency generator is called an \emph{adjacency basis} of $G$, and its cardinality  the \emph{adjacency dimension} of $G$,  denoted by $\dim_A(G)$.
Since any adjacency basis is a metric generator, $\dim(G)\le \dim_A(G)$.  Besides, for any connected graph $G$ of diameter at most two, $\dim_A(G)=\dim(G)$. Moreover, $S$ is an adjacency generator for $G$ if and only if $S$ is an adjacency generator for its complement $\overline{G}$. This is justified by the fact that given an adjacency generator $S$ for $G$, it holds that for every $x,y\in V- S$ there exists $s\in S$ such that $s$ is adjacent to exactly one of $x$ and $y$, and this property holds in $\overline{G}$. Thus, $\dim_A(G)=\dim_A(\overline{G}).$

This concept has been studied further by Fernau and Rodr\'{i}guez-Vel\'{a}zquez in \cite{RV-F-2013,MR3218546} where they showed
that the metric dimension of the corona product of a graph of order $n$ and some non-trivial graph $H$ equals $n$ times the adjacency dimension of $H$. As a consequence of this strong relation they showed that the problem of computing the adjacency dimension is NP-hard.

As pointed out in \cite{RV-F-2013,MR3218546}, 
any  adjacency generator of a graph $G=(V,E)$ is  also a metric generator in a suitably chosen metric space. Given a positive integer $t$,  we define the distance function $d_{G,t}:V\times V\rightarrow \mathbb{N}\cup \{0\}$, where
\begin{equation*}\label{distinguishAdj}
d_{G,t}(x,y)=\min\{d_G(x,y),t\}.
\end{equation*}
Then any metric generator for $(V,d_{G,t})$ is a metric generator for $(V,d_{G,t+1})$ and, as a consequence, the metric dimension of $(V,d_{G,t+1})$  is less than or equal to the metric dimension of $(V,d_{G,t})$. In particular, the metric dimension of $(V,d_{G,1})$ is equal to $|V|-1$,  the metric dimension of $(V,d_{G,2})$ is equal to $\dim_A(G)$ and, if $G$ has diameter $D(G)$, then $d_{G,D(G)}=d_G$ and so  the metric dimension of  $(V,d_{G,D(G)})$  is equal to $\dim(G)$.
Notice that when using the metric $d_{G,t}$ the concept of metric generator needs not be restricted to the case of connected graphs\footnote{For any pair of vertices $x,y$ belonging to different connected components of $G$ we can assume that $d_G(x,y)=\infty$ and so $d_{G,t}(x,y)=t$ for any $t$ greater than or equal to the maximum diameter of a connected component of $G$.}. 

Let ${\mathcal G}=\{G_1,G_2,...,G_k\}$ be a family  of (not necessarily edge-disjoint) connected graphs $G_i=(V,E_i)$ with common vertex set $V$ (the union of whose edge sets is not necessarily the complete graph). Ram\'{i}rez-Cruz, Oellermann  and  Rodr\'{i}guez-Vel\'{a}zquez defined in \cite{Ramirez-Cruz-Rodriguez-Velazquez_2014,Ramirez2014}   a \textit{simultaneous metric generator} for ${\mathcal{G}}$ as a set $S\subset V$ such that $S$ is simultaneously a metric generator for each $G_i$. A smallest simultaneous metric generator for ${\mathcal{G}}$ is  a \emph{simultaneous metric basis} of ${\mathcal{G}}$, and
its cardinality the \emph{simultaneous metric dimension} of ${\mathcal{G}}$, is denoted by $\Sd({\mathcal{G}})$ or explicitly by $\Sd( G_1,G_2,...,G_k )$. By analogy, we defined in \cite{Ramirez_Estrada_Rodriguez_2015} the concept of \emph{simultaneous adjacency generator} for ${\cal G}$,  \emph{simultaneous adjacency basis} of ${\cal G}$ and  the \emph{simultaneous adjacency dimension} of ${\cal G}$, denoted by $\Sd_A({\cal G})$ or explicitly by $\Sd_A(G_1,G_2,...,G_t)$.  For instance, the set $\{1,3,6,7,8\}$ is a simultaneous adjacency basis of the family ${\cal G}=\{G_1,G_2,G_3\}$ shown in Figure  \ref{ExSimultaneousAdjBasis},  while the set $\{1,6,7,8\}$ is a simultaneous metric basis, so $\Sd_A({\cal G})=5$ and $\Sd({\cal G})=4$.

\begin{figure}[h]
\begin{center}

\begin{tikzpicture}[inner sep=0.7mm, place/.style={circle,draw=black!40,
fill=white,thick},dxx/.style={circle,draw=black!99,fill=black!99,thick},fplace/.style={circle,draw=black!40,
fill=black!40,thick},
transition/.style={rectangle,draw=black!50,fill=black!20,thick},line width=1pt,scale=0.5]


\coordinate (v19) at (72:2);
\coordinate (v18) at (144:2);
\coordinate (v17) at (216:2);
\coordinate (v16) at (288:2);
\coordinate (v15) at (0,0);
\coordinate (v14) at (2,0);
\coordinate (v13) at (4,0);
\coordinate (v12) at (6,0);
\coordinate (v11) at (8,0);

\draw[black!40] (v11) -- (v12) -- (v13) -- (v14) -- (v15) -- (v19);
\draw[black!40] (v15) -- (v18);
\draw[black!40] (v15) -- (v17);
\draw[black!40] (v15) -- (v16);

\node [dxx] at (v11) {};
\node at (8.5,-0.5) {$1$};
\node [place] at (v12) {};
\node at (6.5,-0.5) {$2$};
\node [fplace] at (v13) {};
\node at (4.5,-0.5) {$3$};
\node [place] at (v14) {};
\node at (2.5,-0.5) {$4$};
\node [place] at (v15) {};
\node at (0.7,-0.5) {$5$};
\node [dxx] at (v16) {};
\node at (288:2.7) {$6$};
\node [dxx] at (v17) {};
\node at (216:2.7) {$7$};
\node [dxx] at (v18) {};
\node at (144:2.7) {$8$};
\node [place] at (v19) {};
\node at (72:2.7) {$9$};

\coordinate [label=center:{$G_1$}] (G1) at (4,-2);


\coordinate (v29) at (13,2);
\coordinate (v28) at (11,0);
\coordinate (v27) at (13,-2);
\coordinate (v26) at (13,0);
\coordinate (v25) at (15,0);
\coordinate (v24) at (17,0);
\coordinate (v23) at (19,0);
\coordinate (v22) at (21,0);
\coordinate (v21) at (23,0);

\draw[black!40] (v21) -- (v22) -- (v23) -- (v24) -- (v25) -- (v26) -- (v29);
\draw[black!40] (v26) -- (v28);
\draw[black!40] (v26) -- (v27);

\node [dxx] at (v21) {};
\node at (23.5,-0.5) {$1$};
\node [place] at (v22) {};
\node at (21.5,-0.5) {$2$};
\node [fplace] at (v23) {};
\node at (19.5,-0.5) {$3$};
\node [place] at (v24) {};
\node at (17.5,-0.5) {$4$};
\node [place] at (v25) {};
\node at (15.5,-0.5) {$5$};
\node [dxx] at (v26) {};
\node at (13.5,-0.5) {$6$};
\node [dxx] at (v27) {};
\node at (13,-2.7) {$7$};
\node [dxx] at (v28) {};
\node at (10.3,0) {$8$};
\node [place] at (v29) {};
\node at (13,2.7) {$9$};

\coordinate [label=center:{$G_2$}] (G2) at (18,-2);

\end{tikzpicture}

\begin{tikzpicture}[inner sep=0.7mm, place/.style={circle,draw=black!40,
fill=white,thick},dxx/.style={circle,draw=black!99,fill=black!99,thick},fplace/.style={circle,draw=black!40,
fill=black!40,thick},
transition/.style={rectangle,draw=black!50,fill=black!20,thick},line width=1pt,scale=0.5]


\coordinate (v19) at (120:2);
\coordinate (v18) at (240:2);
\coordinate (v17) at (0,0);
\coordinate (v16) at (2,0);
\coordinate (v15) at (4,0);
\coordinate (v14) at (6,0);
\coordinate (v13) at (8,0);
\coordinate (v12) at (10,0);
\coordinate (v11) at (12,0);

\draw[black!40] (v11) -- (v12) -- (v13) -- (v14) -- (v15) -- (v16) -- (v17) -- (v19);
\draw[black!40] (v17) -- (v18);

\node [dxx] at (v11) {};
\node at (12.5,-0.5) {$1$};
\node [place] at (v12) {};
\node at (10.5,-0.5) {$2$};
\node [fplace] at (v13) {};
\node at (8.5,-0.5) {$3$};
\node [place] at (v14) {};
\node at (6.5,-0.5) {$4$};
\node [place] at (v15) {};
\node at (4.5,-0.5) {$5$};
\node [dxx] at (v16) {};
\node at (2.5,-0.5) {$6$};
\node [dxx] at (v17) {};
\node at (0.5,-0.5) {$7$};
\node [dxx] at (v18) {};
\node at (240:2.7) {$8$};
\node [place] at (v19) {};
\node at (120:2.7) {$9$};

\coordinate [label=center:{$G_3$}] (G1) at (6,-2);

\end{tikzpicture}

\end{center}
\caption{The set $\{1,3,6,7,8\}$ is a simultaneous adjacency basis of $\{G_1,G_2,G_3\}$,  whereas $\{1,6,7,8\}$ is a simultaneous metric basis.}
\label{ExSimultaneousAdjBasis}
\end{figure}
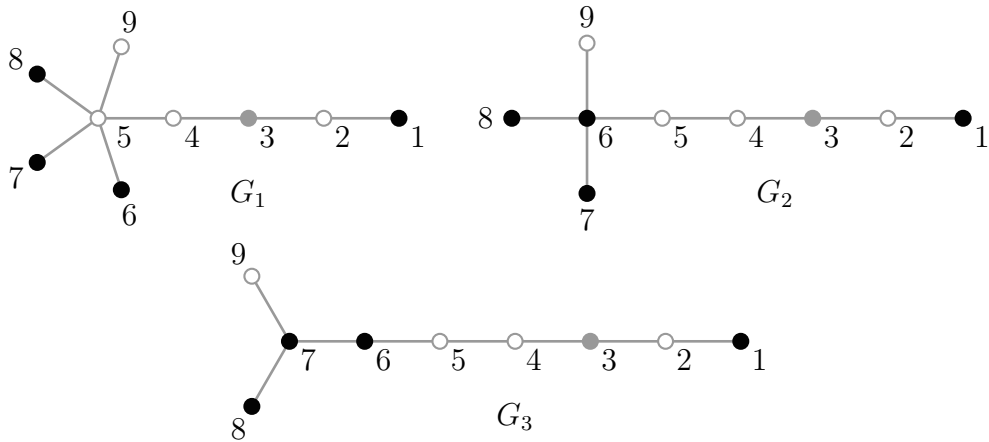

The study of simultaneous parameters in graphs was introduced by Brigham and  Dutton in \cite{Brigham1990}, where they studied simultaneous  domination. 
This should not be  confused with studies on families sharing a constant value on a parameter, for instance the study presented in \cite{IMRAN2013}, where several graph families such that all members have a constant metric dimension are studied, enforcing no constraints regarding whether all members share a metric basis or not. In particular, the study of the simultaneous metric dimension  was introduced  in \cite{Ramirez-Cruz-Rodriguez-Velazquez_2014,Ramirez2014}, where the authors obtained sharp bounds for this invariant  for general families of graphs and gave closed formulae
or tight bounds for the simultaneous metric dimension of several specific graph families. For a given graph $G$ they described a process for obtaining a lower bound
on the maximum number of graphs in a family containing $G$ that has simultaneous metric dimension equal to $\dim(G)$. Moreover, it was shown that the problem of finding the simultaneous metric dimension of families of trees is NP-hard, even when the metric dimension of individual trees can be efficiently computed. This suggests the usefulness of finding the simultaneous  metric dimension for special classes of graphs or obtaining good bounds on this invariant.
In this paper, we obtain closed formulae for the simultaneous metric (adjacency) dimension of corona product graphs. In particular,  we show that the simultaneous adjacency dimension is an important tool for the study of the  simultaneous metric dimension of corona product graphs.

Throughout the paper, we will use the notation $K_n$, $K_{r,s}$, $C_n$, $N_n$ and $P_n$ for complete graphs, complete bipartite graphs, cycle graphs, empty graphs and path graphs of order $n$, respectively. We use the notation $u \sim v$ if $u$ and $v$ are adjacent and $G \cong H$ if $G$ and $H$ are isomorphic graphs. For a vertex $v$ of a graph $G$, $N_G(v)$ will denote the set of neighbours or \emph{open neighbourhood} of $v$ in $G$, \emph{i.e.} $N_G(v)=\{u \in V(G):\; u \sim v\}$. The \emph{closed neighbourhood}, denoted by $N_G[v]$, equals $N_G(v) \cup \{v\}$. If there is no ambiguity, we will simple write $N(v)$ or $N[v]$. Two vertices $x,y\in V(G)$  are \textit{twins} in $G$ if $N_G[x]=N_G[y]$ or $N_G(x)=N_G(y)$. If $N_G[x]=N_G[y]$, they are said to be \emph{true twins}, whereas if $N_G(x)=N_G(y)$ they are said to be \emph{false twins}. We also define $\delta(v)=|N(v)|$ as the degree of vertex $v$, as well as $\delta(G)=\min_{v \in V(G)}\{\delta(v)\}$ and $\Delta(G)=\max_{v \in V(G)}\{\delta(v)\}$. The subgraph induced by a set $S$ of vertices will be denoted  by $\langle S\rangle$, the diameter of a  graph will be denoted by $D(G)$ and the girth by $\mathtt{g}(G)$.
For the remainder of the paper, definitions will be introduced whenever a concept is needed.

\section{The simultaneous adjacency dimension: basic bounds and tools}

The following general bounds on $\Sd_A({\cal G})$ were discussed in \cite{Ramirez_Estrada_Rodriguez_2015}.

\begin{remark}{\rm \cite{Ramirez_Estrada_Rodriguez_2015}}\label{trivialBoundsSimAdjDim}
For any family ${\cal G}=\{G_1,G_2,...,G_t\}$ of  connected graphs on a common vertex set $V$, the following results hold:
\begin{enumerate}[{\rm (i)}]
\item $\Sd_A({\cal G}) \geq \underset{i\in \{1,...,k\}}{\max}\{\dim_A(G_i)\}$.
\item $\Sd_A({\cal G})\geq\Sd({\cal G})$.
\item $\Sd_A({\cal G}) \leq \vert V\vert-1$.
\end{enumerate}
\end{remark}

It was also shown in \cite{Ramirez_Estrada_Rodriguez_2015} that if ${\cal G}$ is graph family   defined on a common vertex set $V$, such that for every pair of different vertices $u,v \in V$ there exists a graph $G \in {\cal G}$ where $u$ and $v$ are twins, then  $\Sd_A({\cal G})=\vert V \vert - 1$. In particular, any family ${\cal G}$ containing a complete graph or an empty graph satisfies $\Sd_A({\cal G})=\vert V \vert - 1$. Moreover, since a graph and its complement have the same adjacency generators, we have that $\Sd_A({\cal G})=\Sd_A(\overline{\cal G})=\Sd_A({\cal G}\cup \overline{\cal G})$, where $\overline{\cal G}=\{\overline{G} :\; G \in {\cal G}\}$.

Let $G=(V,E)$ be a graph and let  $\Perm(V)$ be the set of all permutations of $V$. 
 Given a subset $X\subseteq V$, the stabilizer of $X$ is the set of permutations $${\cal S}(X)=\{f\in \Perm(V): f(x)=x, \; \mbox{\rm for every } x\in X\} .$$  As usual, we denote by $f(X)$ the image of a subset $X$ under $f$, \textit{i.e}., $f(X)=\{f(x):\; x\in X\}$.

Let $G=(V,E)$ be a graph and let $B\subset V$ be a nonempty set. For any  
permutation $f\in {\cal S}(B)$ of $V$ we say that a graph $G'=(V,E')$ belongs to the  family  ${\cal  G}_{B,f}(G)$ if and only if $N_{G'}(x)=f(N_G(x))$ for every $x\in B$. We define the subgraph $\langle B_G\rangle_{w}=(N_G[B],E_{w})$ of $G$, weakly induced by $B$,  where $N_G[B]=\cup_{x\in B}N_G[x]$ and $E_{w}$ is the set of all edges having at least one vertex
in $B$. It was shown in \cite{Ramirez_Estrada_Rodriguez_2015} that $\langle B_G \rangle_w \cong \langle B_{G'} \rangle_w$ for any  $f\in {\cal{S}}(B)$ and any graph  $G'\in  {\cal{G}}_{B,f}(G)$. We define the graph family ${\cal  G}_B(G)$, associated to $B$,  as 
$${\cal  G}_B(G)=\displaystyle\bigcup_{f\in {\cal S}(B)}{\cal  G}_{B,f}(G).$$

\begin{figure}[h]
\begin{center}
\begin{tikzpicture}[inner sep=0.7mm, place/.style={circle,draw=black!40,
fill=white,thick},xx/.style={circle,draw=black!99,fill=black!99,thick},
transition/.style={rectangle,draw=black!50,fill=black!20,thick},line width=1pt,scale=0.5]

\def\n{8}

\def\radius{2.65cm}
\def\lblradius{3.35cm}


\foreach \ind in {1,...,\n}\pgfmathparse{360/\n*\ind}\coordinate (b\ind) at (\pgfmathresult:\radius);

\foreach \ind in {1,...,\n}\pgfmathparse{360/\n*\ind}\coordinate (v0\ind) at (\pgfmathresult:\lblradius);

\draw[black!40] (b1) -- (b2) -- (b3) -- (b4) -- (b5) -- (b6) -- (b7) -- (b8) -- cycle;

\node [xx] at (b1) {};
\node at (v01) {$1$};
\node [place] at (b2) {};
\node at (v02) {$2$};
\node [xx] at (b3) {};
\node at (v03) {$3$};
\node [place] at (b4) {};
\node at (v04) {$4$};
\node [place] at (b5) {};
\node at (v05) {$5$};
\node [place] at (b6) {};
\node at (v06) {$6$};
\node [xx] at (b7) {};
\node at (v07) {$7$};
\node [place] at (b8) {};
\node at (v08) {$8$};

\coordinate [label=center:{$G$}] (G) at (0.1,-4.5);


\coordinate [label=center:{$f$}] (f) at (9.1,3);

\coordinate [label=center:{$1 \rightarrow 1$}] (f111) at (7.1,1);
\coordinate [label=center:{$2 \rightarrow 6$}] (f126) at (7.1,0);
\coordinate [label=center:{$3 \rightarrow 3$}] (f133) at (7.1,-1);
\coordinate [label=center:{$4 \rightarrow 8$}] (f148) at (7.1,-2);

\coordinate [label=center:{$5 \rightarrow 2$}] (f152) at (11.1,1);
\coordinate [label=center:{$6 \rightarrow 4$}] (f164) at (11.1,0);
\coordinate [label=center:{$7 \rightarrow 7$}] (f177) at (11.1,-1);
\coordinate [label=center:{$8 \rightarrow 5$}] (f185) at (11.1,-2);


\foreach \ind in {1,...,\n}\pgfmathparse{360/\n*\ind}\coordinate [yshift=-4.5cm] (g1\ind) at (\pgfmathresult:\radius);

\foreach \ind in {1,...,\n}\pgfmathparse{360/\n*\ind}\coordinate [yshift=-4.5cm] (v1\ind) at (\pgfmathresult:\lblradius);

\draw[black!40] (g11) -- (g12) -- (g13) -- (g14) -- (g15) -- (g16) -- (g17) -- (g18) -- cycle;
\draw[black!40] (g12) -- (g15) -- (g18);

\node [xx] at (g11) {};
\node at (v11) {$1$};
\node [place] at (g12) {};
\node at (v12) {$6$};
\node [xx] at (g13) {};
\node at (v13) {$3$};
\node [place] at (g14) {};
\node at (v14) {$8$};
\node [place] at (g15) {};
\node at (v15) {$2$};
\node [place] at (g16) {};
\node at (v16) {$4$};
\node [xx] at (g17) {};
\node at (v17) {$7$};
\node [place] at (g18) {};
\node at (v18) {$5$};

\coordinate [label=center:{$H_1$}] (H1) at (0.1,-13.5);


\foreach \ind in {1,...,\n}\pgfmathparse{360/\n*\ind}\coordinate [xshift=3.8cm,yshift=-4.5cm] (g2\ind) at (\pgfmathresult:\radius);

\foreach \ind in {1,...,\n}\pgfmathparse{360/\n*\ind}\coordinate [xshift=3.8cm,yshift=-4.5cm] (v2\ind) at (\pgfmathresult:\lblradius);

\draw[black!40] (g26) -- (g27) -- (g28) -- (g21) -- (g22) -- (g23) -- (g24) -- (g25);

\node [xx] at (g21) {};
\node at (v21) {$1$};
\node [place] at (g22) {};
\node at (v22) {$6$};
\node [xx] at (g23) {};
\node at (v23) {$3$};
\node [place] at (g24) {};
\node at (v24) {$8$};
\node [place] at (g25) {};
\node at (v25) {$2$};
\node [place] at (g26) {};
\node at (v26) {$4$};
\node [xx] at (g27) {};
\node at (v27) {$7$};
\node [place] at (g28) {};
\node at (v28) {$5$};

\coordinate [label=center:{$H_2$}] (H2) at (7.7,-13.5);


\foreach \ind in {1,...,\n}\pgfmathparse{360/\n*\ind}\coordinate [xshift=7.6cm,yshift=-4.5cm] (g3\ind) at (\pgfmathresult:\radius);

\foreach \ind in {1,...,\n}\pgfmathparse{360/\n*\ind}\coordinate [xshift=7.6cm,yshift=-4.5cm] (v3\ind) at (\pgfmathresult:\lblradius);

\draw[black!40] (g35) -- (g36) -- (g37) -- (g38) -- (g31) -- (g32) -- (g33) -- (g34);
\draw[black!40] (g36) -- (g34);
\draw[black!40] (g36) -- (g32);
\draw[black!40] (g36) -- (g38);

\node [xx] at (g31) {};
\node at (v31) {$1$};
\node [place] at (g32) {};
\node at (v32) {$6$};
\node [xx] at (g33) {};
\node at (v33) {$3$};
\node [place] at (g34) {};
\node at (v34) {$8$};
\node [place] at (g35) {};
\node at (v35) {$2$};
\node [place] at (g36) {};
\node at (v36) {$4$};
\node [xx] at (g37) {};
\node at (v37) {$7$};
\node [place] at (g38) {};
\node at (v38) {$5$};

\coordinate [label=center:{$H_3$}] (H3) at (15.3,-13.5);


\foreach \ind in {1,...,\n}\pgfmathparse{360/\n*\ind}\coordinate [xshift=11.4cm,yshift=-4.5cm] (g4\ind) at (\pgfmathresult:\radius);

\foreach \ind in {1,...,\n}\pgfmathparse{360/\n*\ind}\coordinate [xshift=11.4cm,yshift=-4.5cm] (v4\ind) at (\pgfmathresult:\lblradius);

\draw[black!40] (g41) -- (g42) -- (g43) -- (g44) -- (g45) -- (g46) -- (g47) -- (g48) -- cycle;
\draw[black!40] (g46) -- (g44);
\draw[black!40] (g44) -- (g48);
\draw[black!40] (g42) -- (g44);
\draw[black!40] (g42) -- (g48);

\node [xx] at (g41) {};
\node at (v41) {$1$};
\node [place] at (g42) {};
\node at (v42) {$6$};
\node [xx] at (g43) {};
\node at (v43) {$3$};
\node [place] at (g44) {};
\node at (v44) {$8$};
\node [place] at (g45) {};
\node at (v45) {$2$};
\node [place] at (g46) {};
\node at (v46) {$4$};
\node [xx] at (g47) {};
\node at (v47) {$7$};
\node [place] at (g48) {};
\node at (v48) {$5$};

\coordinate [label=center:{$H_4$}] (H4) at (22.9,-13.5);

\end{tikzpicture}

\end{center}
\caption{A subfamily ${\cal H}$ of ${\cal G}_B(C_8)$, where $B=\{1,3,7\}$. For every $H_i \in {\cal H}$, $\dim_A(H_i)=\dim_A(C_8)=3$. Moreover, $B$ is a simultaneous adjacency basis of ${\cal H}$, so $\Sd_A({\cal H})=3$.}
\label{FigSubFamilyC8}
\end{figure}
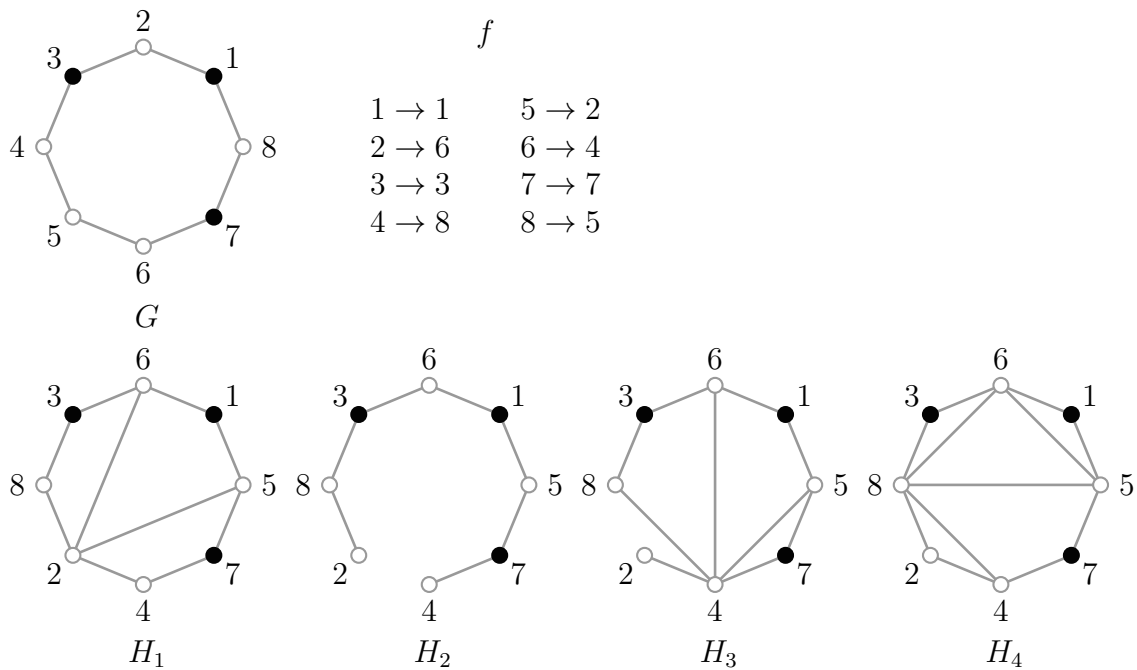

The following result shows that, given a graph $G$ and an adjacency basis $B$ of $G$, we can obtaining large families of graphs having $B$ as a simultaneous adjacency generator. 

\begin{theorem}{\rm \cite{Ramirez_Estrada_Rodriguez_2015}}\label{dimAPermFamily}
Any  adjacency basis $B$    of a graph $G$ is  a simultaneous adjacency generator for  any family of graphs ${\cal  H} \subseteq {\cal  G}_B(G)$. Moreover, if   $G\in {\cal  H}$, then
 $$\Sd_A({\cal  H})=\dim_A(G).$$
\end{theorem}

To illustrate this, Figure~\ref{FigSubFamilyC8} shows a graph family ${\cal H}=\{H_1,H_2,H_3,H_4\} \subseteq {\cal G}_B(C_8)$, where $B=\{1,3,7\}$ and $\Sd_A({\cal  H})=\dim_A(C_8).$

\section{Results on families of corona product graphs}

Let $G$ be a graph  of order $n$ and $H$ be a graph. The \emph{corona product} of $G$ and $H$, denoted by  $G\odot H$, was defined in  \cite{Frucht1970} as the graph obtained from $G$ and $H$ by taking one copy of $G$ and $n$ copies of $H$ and joining by an edge each vertex from the $i$-th copy of $H$ with the $i$-th vertex of $G$. The reader is referred to \cite{RV-F-2013,MR3218546,Frucht1970,Yero2011,Feng2012a,
BarraganRamirez201427,Casablanca2011,Estrada-Moreno2013corona,
Furmanczyk2013,Yero2013d,Kuziak2013,Yarahmadi2012,Yero2011b,MR3267826,
MR3239626,MR3198333,MR3194036} for some known results on corona product graphs. 

In order to present our results on the simultaneous metric (adjacency) dimension of graph families composed by corona product graphs, we need to introduce some additional notation. 
For a family ${\cal G}$ of connected non-trivial graphs defined on a common vertex set $V$ and a family ${\cal H}$ of non-trivial graphs defined on a common vertex set $V'$, we define the family $${\cal G}\odot{\cal H}=\{G  \odot H :\; G\in {\cal G}\; {\rm and  }\; H\in {\cal H}\}.$$
In particular, if ${\cal G}=\{G\}$, we will use the notation $G \odot {\cal H}$, whereas if ${\cal H}=\{H\}$, we will use the notation ${\cal G} \odot H$. 

Given $G\in \mathcal{G}$ and $H\in \mathcal{H}$, we denote by  $H_i=(V_i',E_i)$  the subgraph of $G\odot H$ corresponding to the $i$-th copy of $H$. Notice that for any $i\in V$ the graph $H_i$, which is isomorphic to $H$,  does not depend on $G$. 
Hence, the graphs in  ${\cal G}\odot{\cal H}$ are defined on the vertex set $V\cup \left(\displaystyle \bigcup_{i\in V}V_i'\right)$.
Analogously, for every  $i\in V$  we define the graph family $${\cal H}_i=\{H_i=(V_i',E_i): \; H\in {\cal H}\}.$$
Also, given a set $W\subset V'$ and $i\in V$, we denote by $W_i$ the subset of $V_i'$ corresponding to $W$. To clarify this notation, Figure~\ref{figExplVVprime} shows the graph $C_4 \odot (K_1 \cup K_2)$. In the figure, $V=\{1,2,3,4\}$ and $V'=\{a,b,c\}$, whereas $V'_i=\{a_i,b_i,c_i\}$ for $i \in \{1,2,3,4\}$.

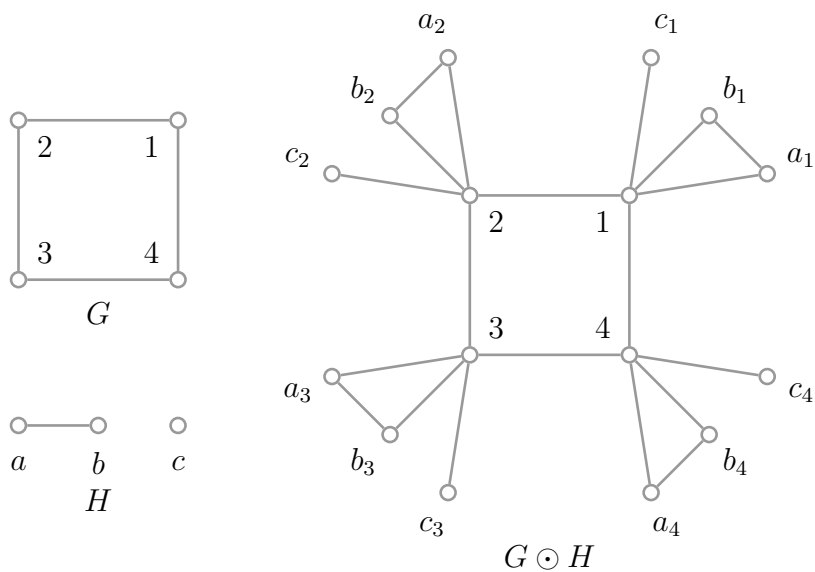
\begin{figure}[!ht]
\centering
\begin{tikzpicture}[transform shape, inner sep = .7mm,line width=1pt]
\def\radius{1.5};
\def\alfa{20};
\foreach \x in {1,...,4}
{
\pgfmathparse{int(45+90*(\x-1))};
\node [draw=black!40, shape=circle, fill=white] (\x) at (\pgfmathresult:\radius) {};
\node  at (\pgfmathresult:\radius-.5) {$\x$};
\node [draw=black!40, shape=circle, fill=white] (b\x) at (\pgfmathresult:2*\radius) {};
\node  at (\pgfmathresult:2*\radius+.5) {$b_\x$};
\pgfmathsetmacro{\r}{2*\radius/cos(\alfa)};
\pgfmathparse{int(45-\alfa+90*(\x-1))};
\node [draw=black!40, shape=circle, fill=white] (a\x) at (\pgfmathresult:\r) {};
\node at (\pgfmathresult:\r+.5) {$a_\x$};
\pgfmathparse{int(45+\alfa+90*(\x-1))};
\node [draw=black!40, shape=circle, fill=white] (c\x) at (\pgfmathresult:\r) {};
\node at (\pgfmathresult:\r+.5) {$c_\x$};
}
\draw [black!40](1) -- (2) -- (3) -- (4) -- (1);
\foreach \x in {1,...,4}
{
\draw [black!40](\x) -- (a\x);
\draw [black!40](\x) -- (b\x);
\draw [black!40](\x) -- (c\x);
\draw [black!40](a\x) -- (b\x);
}
\node at (0,-2.5*\radius cm) {$G\odot H$};
\node (centernew) at (-4*\radius cm, 1cm) {};
\foreach \x in {1,...,4}
{
\pgfmathparse{int(45+90*(\x-1))};
\node [draw=black!40, shape=circle, fill=white] (c\x) at ([shift=({\pgfmathresult:\radius})]centernew) {};
\node at ([shift=({\pgfmathresult:\radius-.5})]centernew) {$\x$};
}
\draw [black!40](c1) -- (c2) -- (c3) -- (c4) -- (c1);
\node at ([yshift=-\radius cm]centernew) {$G$};
\node [draw=black!40, shape=circle, fill=white] (b) at ([shift=({0,-2*\radius cm})]centernew) {};
\node at ([yshift=-1 cm]b) {$H$};
\pgfmathparse{\radius/sqrt(2)};
\node at ([yshift=-.5 cm]b) {$b$};
\node [draw=black!40, shape=circle, fill=white] (a) at ([shift=({-\pgfmathresult,0})]b) {};
\node at ([yshift=-.5 cm]a) {$a$};
\node [draw=black!40, shape=circle, fill=white] (c) at ([shift=({\pgfmathresult,0})]b) {};
\node at ([yshift=-.5 cm]c) {$c$};
\draw [black!40](a) -- (b);
\end{tikzpicture}
\caption{The graph $G \odot H$, where $G \cong C_4$ and $H \cong K_1 \cup K_2$.}
\label{figExplVVprime}
\end{figure}

\subsection{Simultaneous metric dimension}

We first introduce a useful relation between the metric generators of two corona product graphs with a common second factor, which allows to determine the simultaneous metric dimension of several families of corona product graphs through the study of the metric dimension of a specific corona product graph.

\begin{theorem}\label{foundMainLemmaCorona}
Let $G_1$ and $G_2$ be two connected non-trivial graphs on a common vertex set and let $H$ be a non-trivial graph. Then any metric generator for $G_1 \odot H$ is a metric generator for $G_2 \odot H$.
\end{theorem}

\begin{proof}
Let $V$ be the vertex set of $G_1$ and $G_2$ and let $V'$ be the vertex set of $H$. We claim that any metric generator $B$ for $G_1\odot H$   is a metric generator for $G_2\odot H$. To see this, we differentiate the following three cases for two different vertices $x,y \in V(G_2\odot H)-B$.

\begin{enumerate}
\item $x,y\in V_i'$. Since no vertex belonging to $B-V_i'$ distinguishes the pair $x,y$ in $G_1 \odot H$, there must exist $u\in V_i'\cap B$ which distinguishes them. This vertex $u$ also distinguishes $x$ and $y$ in $G_2 \odot H$.
\item Either  $x\in V_i'$ and $y\in V_j'$ or $x=i$ and $y\in V_j'$, where $i\ne j$. For these two possibilities we take $u\in B\cap V_i'$ and we conclude that $d_{G_2\odot H}(x,u)\le 2 \neq 3\le d_{G_2\odot H}(y,u)$.
\item $x=i$ and $y=j$. In this case for $u\in B\cap V_i'$   we have $d_{G_2\odot H}(x,u)=1 \neq 2\le d_{G_2\odot H}(y,u)$.
\end{enumerate}

In conclusion, $B$ is a metric generator for $G_2\odot H$.
\end{proof}

The following result is a direct consequence of Theorem~\ref{foundMainLemmaCorona}.

\begin{corollary}\label{MainLemmaCorona}
Let ${\cal G}$ be a family of connected non-trivial graphs on a common vertex set and let ${\cal H}$ be a family of non-trivial graphs on a common vertex set. Then, for any $G\in \mathcal{G}$, $$\Sd({\cal G} \odot {\cal H})=\Sd(G \odot {\cal H}).$$
\end{corollary}

The following result, obtained in \cite{RV-F-2013}, provides a strong link between the metric dimension of the corona product of two graphs and the adjacency dimension of the second graph involved in the product operation.

\begin{theorem}{\rm \cite{RV-F-2013}} \label{mainTheoremDimAdjDim}
For any connected graph $G$ of order $n\ge 2$ and any non-trivial graph $H$, $$\dim(G\odot H)=n \cdot \dim_A(H).$$
\end{theorem}

We now present a generalisation of Theorem~\ref{mainTheoremDimAdjDim} to deal with graph families.

\begin{theorem}\label{relSdCoronaFamSdAFam}
For any family ${\cal G}$ composed by connected non-trivial graphs on a common vertex set $V$ and any family ${\cal H}$ composed by non-trivial graphs on a common vertex set, $$\Sd({\cal G} \odot {\cal H}) = \vert V \vert \cdot \Sd_A({\cal H}).$$
\end{theorem}

\begin{proof}
Throughout the proof we  consider two arbitrary graphs   $G  \in {\cal G} $ and $ H \in  {\cal H}$. Let $B$ be a simultaneous metric basis of ${\cal G} \odot {\cal H}$ and  let $B_i=B \cap V'_i$. Clearly, $B_i \cap B_j=\emptyset$ for every $i \neq j$. Since no pair of vertices $x,y \in H_i$ is distinguished by any vertex $v \in B_j$, $i \neq j$, we have that $B_i$ is an adjacency generator for $H_i$. Hence, the set $B'\subset V'$ corresponding to $B_i\subset V'_i$ is an adjacency generator for $H$ and, since  $B'$  does not depend on the election of $H$, it  is a simultaneous adjacency generator for $\mathcal{H}$ and, as a result, $$\Sd({\cal G} \odot {\cal H})=|B|\ge \displaystyle\sum_{i\in V}|B_i| =|V||B'|\geq |V|\cdot\Sd_A({\cal H}).$$

Now, let $W$ be a simultaneous adjacency basis of ${\cal H}$ and let $W_i=W \cap V'_i$. By analogy to  the proof of Theorem \ref{mainTheoremDimAdjDim} we   see that $S=\displaystyle \bigcup_{i\in V} W_i$ is a metric generator for $G \odot H$. Since  $S$ does not depend on the election of $G$ and $H$, it is a simultaneous metric generator for ${\cal G}\odot{\cal H}$ and  so $$\Sd({\cal G} \odot {\cal H}) \leq |S| = \displaystyle\sum_{i\in V} |W_i| = |V| \cdot \Sd_A({\cal H}).$$
Therefore,  the equality holds.
\end{proof}

The following result is a direct consequence of Theorems~\ref{dimAPermFamily} and~\ref{relSdCoronaFamSdAFam}.

\begin{proposition}
Let ${\cal G}$ be a family of connected non-trivial graphs on a common vertex set $V$. Let $H$ be a non-trivial graph and let $B$ be an adjacency basis of $H$. Then, for every ${\cal H} \subseteq {\cal G}_B(H)$ such that $H \in {\cal H}$, $$\Sd({\cal G}\odot{\cal H})=|V|\cdot\dim_A(H).$$
\end{proposition}

\subsection{Simultaneous adjacency dimension}

Given a    family ${\cal G}$ of connected non-trivial graphs on a common vertex set $V$ and  a family ${\cal H}$    of non-trivial graphs on a common vertex set,  Remark~\ref{trivialBoundsSimAdjDim} and Theorem~\ref{relSdCoronaFamSdAFam} lead to 
\begin{equation}\label{TrivialLowerBound}
\Sd_A({\cal G} \odot {\cal H})\geq \Sd({\cal G} \odot {\cal H})=|V|\cdot\Sd_A({\cal H}). 
\end{equation}
Therefore, there exists an integer $f({\cal G},{\cal H})\ge 0$ such that
\begin{equation}\label{MainFormulaAdjDim}
\Sd_A({\cal G} \odot {\cal H})=|V|\cdot\Sd_A({\cal H})+f({\cal G},{\cal H}).
\end{equation} 

If ${\cal G}=\{G\}$ or ${\cal H}=\{H\}$, we will use the notations $f(G,{\cal H})$, $f({\cal G},H)$ or $f(G,H)$, as convenient. It is easy to check that for any simultaneous adjacency basis $W$ of $\mathcal{H}$ and any $i\in V$, the set $(V-\{i\})\cup \left(\displaystyle\bigcup_{j\in V}W_j\right)$ is a simultaneous adjacency generator for ${\cal G} \odot {\cal H}$, where $W_j$ is the subset of $V'_j$   corresponding to $W\subset V'$. Hence, 
 \begin{equation}\label{NaturalBoundsOnf}0\le f({\cal G},{\cal H}) \le |V|-1.
 \end{equation}

From now on, our goal is to determine the value of $f({\cal G},{\cal H})$ under different sets of conditions. We begin by pointing out a useful fact which we will use throughout the remainder of this section. Let $B$ be a simultaneous adjacency basis of ${\cal G}\odot{\cal H}$, and let $B_i=B \cap V'_i$. The following observation is a consequence of the fact that for any graph $G \odot H \in {\cal G}\odot{\cal H}$ and $i \in V$, no vertex in $B-B_i$ is able to distinguish two vertices in $V_i'$.

\begin{remark}\label{BiSimGenHi}
Let ${\cal G}$ be a family of connected non-trivial graphs on a common vertex set $V$ and let ${\cal H}$ be a family of non-trivial graphs on a common vertex set $V'$. Let $B$ be a simultaneous adjacency basis of ${\cal G}\odot{\cal H}$ and let $B_i=B\cap V'_i$ for every $i \in V$. Then, $B_i$ is a simultaneous adjacency generator for ${\cal H}_i$.
\end{remark}

Now, consider the following known result where $f(G,H)=0$.

\begin{theorem}{\rm \cite{RV-F-2013}}\label{TheOnlyPosibilitiesDimAdjCorona(a)}
Let $G$ be a  connected graph of order $n\ge 2$ and let $H$ be a non-trivial graph. 
If there exists an adjacency basis $S$ of $H$, which is also a dominating set, and if for every $v\in V(H)-S$,  it is satisfied that $S\not\subseteq N_H(v)$, then  
$$\dim_A(G\odot H)=n\cdot \dim_A(H).$$
\end{theorem}

As the next result shows, Theorem \ref{TheOnlyPosibilitiesDimAdjCorona(a)}  can be generalised to the case of families of the form ${\cal G} \odot {\cal H}$. To that end, recall the notion of simultaneous domination which, as we mentioned previously, was introduced in \cite{Brigham1990}. On a graph family ${\cal G}$, defined on a common vertex set $V$, a set $M \subseteq V$ is a \emph{simultaneous dominating set} if it is a dominating set of every graph $G \in {\cal G}$.
 
\begin{theorem}\label{simAdjDimFamCoronaCase1}
Let ${\cal G}$ be a family of connected non-trivial graphs on a common vertex set $V$ and let ${\cal H}$ be a family of non-trivial graphs on a common vertex set $V'$. If there exists a simultaneous adjacency basis $B$ of ${\cal H}$ which is also a simultaneous dominating set and satisfies $B \nsubseteq N_H(v)$ for every $H \in {\cal H}$ and every $v \in V'$, then $$\Sd_A({\cal G} \odot {\cal H})=|V|\cdot\Sd_A({\cal H}).$$
\end{theorem}

\begin{proof}
By (\ref{TrivialLowerBound}) we only need to show that $\Sd_A({\cal G} \odot {\cal H})\leq |V|\cdot\Sd_A({\cal H})$. To this end, assume that  $B$ is  a simultaneous adjacency basis of ${\cal H}$ which is a simultaneous dominating set of ${\cal H}$ and satisfies $B \nsubseteq N_H(v)$ for every $H \in {\cal H}$ and every $v \in V'$. Consider an arbitrary graph $G \odot H \in {\cal G}\odot{\cal H}$ and let $B_i=B \cap V'_i$, for every $i\in V$. By analogy to  the proof of Theorem \ref{TheOnlyPosibilitiesDimAdjCorona(a)}   we see that  $S=\displaystyle{\bigcup_{i\in V}}B_i$ is an adjacency generator for $G \odot H$ and, since $S$ does not depend on the election of $G$ and $H$,  it is a simultaneous adjacency generator for ${\cal G} \odot {\cal H}$. Thus, $\Sd_A({\cal G} \odot {\cal H}) \leq |S|= |V|\cdot\Sd_A({\cal H})$, and the equality holds.
\end{proof}

In order to analyse special cases of Theorem~\ref{simAdjDimFamCoronaCase1}, we introduce the following auxiliary results.

\begin{lemma}{\rm \cite{Ramirez_Estrada_Rodriguez_2015}}\label{LemmaDiameter>6orPathorCycle}
Let $G$ be a connected graph. If $D(G)\ge 6$ or $G\in \{P_n,C_n\}$ for $n\ge 7$, or $G$ is a graph of girth $\mathtt{g}(G)\ge 5$ and minimum degree $\delta(G)\ge 3$ then for every adjacency generator $B$ for $G$ and every $v\in V(G)$, $B\not \subseteq N_G(v).$
\end{lemma}

\begin{lemma}{\rm \cite{Ramirez_Estrada_Rodriguez_2015}}\label{lemmaExistAdjBDominatingCyclePath}
Let $P_n$ and $C_n$ be a path and a cycle graph of order $n \ge 7$. If $n \not\equiv 1 \mod 5$ and $n \not\equiv 3 \mod 5$, then there exist adjacency bases of $P_n$ and $C_n$ that are dominating sets.
\end{lemma}

\begin{lemma}{\rm \cite{JanOmo2012}}
\label{LemmaAdimPathCycle}
For any integer $n\ge 4$,  
 $$\dim_A(C_n)=\dim_A( P_n)= \left\lfloor\frac{2n+2}{5}\right\rfloor.$$
\end{lemma}

\begin{proposition}\label{simAdjDimFamCoronaCase1GenPartCase1}
Let ${\cal G}$ be a family of connected non-trivial graphs on a common vertex set $V$. Let $P_n$ be a path graph of order $n \ge 7$ such that $n\not \equiv 1 \bmod 5$ and  $n\not\equiv 3 \bmod 5$, and let $C_n$ be the cycle graph obtained from $P_n$ by joining its leaves by an edge. Let $B$ be an adjacency basis of $P_n$ and $C_n$ which is also a dominating set of both. Then, for every ${\cal H}\subseteq {\cal G}_B(P_n) \cup {\cal G}_B(C_n)$ such that $P_n \in {\cal H}$ or $C_n \in {\cal H}$, $$\Sd_A({\cal G} \odot {\cal H})=|V|\cdot\left\lfloor\frac{2n+2}{5}\right\rfloor.$$
\end{proposition}

\begin{proof}
The existence of $B$ is a consequence of Lemma~\ref{lemmaExistAdjBDominatingCyclePath}. Since $P_n \in {\cal H}$ or $C_n \in {\cal H}$, by Theorem \ref{dimAPermFamily} we deduce that $B$ is a simultaneous adjacency basis of ${\cal H}$. Let $V'=V(P_n)=V(C_n)$. By the definition of ${\cal G}_B$, we have that $\displaystyle{\bigcup_{v \in B}}N_{H}(v)=\displaystyle{\bigcup_{v \in B}}N_{P_n}(v)=V'$ or $\displaystyle{\bigcup_{v \in B}}N_{H}(v)=\displaystyle{\bigcup_{v \in B}}N_{C_n}(v)=V'$ for every $H \in {\cal H}$, so $B$ is a dominating set of every $H \in {\cal H}$. Moreover, by Lemma~\ref{LemmaDiameter>6orPathorCycle}, we have that $B \nsubseteq N_{P_n}(v)$ and $B \nsubseteq N_{C_n}(v)$ for every $v \in V'$. Furthermore, by the definition of ${\cal G}_B$, we have that $B \cap N_{H}(v)=B \cap N_{P_n}(v)$ or $B \cap N_{H}(v)=B \cap N_{C_n}(v)$ for every $H \in {\cal H}$ and every $v\in V'$, so $B \nsubseteq N_{H}(v)$ for every $H \in {\cal H}$ and every $v\in V'$. In consequence, the result follows from Lemma \ref{LemmaAdimPathCycle} and Theorems~\ref{dimAPermFamily} and~\ref{simAdjDimFamCoronaCase1}.
\end{proof}

In order to show some cases where $f(\mathcal{G},\mathcal{H})=|V|-1$, we present the following result. 

\begin{theorem}\label{simAdjDimFamCoronaCase2}
Let ${\cal G}$ be a family of connected non-trivial graphs on a common vertex set $V$ and let ${\cal H}$ be a family of non-trivial graphs on a common vertex set. If for every simultaneous adjacency basis $B$ of ${\cal H}$ there exists $H \in {\cal H}$ where $B$ is not a dominating set, then $$\Sd_A({\cal G} \odot {\cal H})=|V| \cdot\Sd_A({\cal H})+|V|-1.$$
\end{theorem}

\begin{proof}
By (\ref{MainFormulaAdjDim}) and (\ref{NaturalBoundsOnf}) we have that $\Sd_A({\cal G} \odot {\cal H})\le |V| \cdot\Sd_A({\cal H})+|V|-1.$
It remains  to prove that $\Sd_A({\cal G} \odot {\cal H})\ge |V| \cdot\Sd_A({\cal H})+|V|-1.$

Let  $U$ be a simultaneous adjacency basis of ${\cal G} \odot {\cal H}$, let $U_i=U \cap V'_i$ and let $U_0=U \cap V$. By Remark~\ref{BiSimGenHi}, $U_i$ is a simultaneous adjacency generator for ${\cal H}_i$ for every $i \in V$. Consider the partition $\{V_1,V_2\}$ of $V$ defined as $$V_1=\{i \in V:\; |U_i|=\Sd_A({\cal H})\}\ \text{and}\ V_2=\{i \in V:\; |U_i|\ge\Sd_A({\cal H})+1\}.$$ For any $i,j \in V_1$, $i\ne j$, we have that there exist a graph $H \in {\cal H}$ and two vertices $x \in V'_i-U_i$ and $y \in V'_j-U_j$ such that $U_i \cap N_{H}(x)=\emptyset$ and $U_j \cap N_{H}(y)=\emptyset$. Thus, $i \in U$ or $j \in U$ and so $|U_0| \geq |V_1|-1$. In conclusion,

\begin{align*}
\Sd_A({\cal G} \odot {\cal H})&=|U_0|+\displaystyle{\sum_{i\in V_1}}|U_i|+\sum_{i\in V_2}|U_i|\\
&\geq(|V_1|-1)+|V_1|\cdot\Sd_A({\cal H})+|V_2|\cdot(\Sd_A({\cal H})+1)\\
&=|V|\cdot\Sd_A({\cal H})+|V|-1.
\end{align*}

Therefore, the result follows.
\end{proof}

Now we treat some specific families for which the previous results hold. We first introduce an auxiliary result.

\begin{lemma}\label{lemmaNonExistAdjBDominatingCyclePath}
Let $P_n$ and $C_n$ be a path and a cycle graph of order $n \ge 7$. If $n \equiv 1 \mod 5$ or $n \equiv 3 \mod 5$, then no adjacency basis of $P_n$ or $C_n$ is a dominating set.
\end{lemma}

\begin{proof}
In $C_n$, consider an adjacency basis $B$ and a path $v_iv_{i+1}v_{i+2}v_{i+3}v_{i+4}$, where the subscripts are taken modulo $n$. If $v_i,v_{i+2} \in B$ and $v_{i+1} \notin B$, then $\{v_{i+1}\}$ is said to be a 1-gap of $B$. Likewise, if $v_i,v_{i+3} \in B$ and $v_{i+1},v_{i+2} \notin B$, then $\{v_{i+1},v_{i+2}\}$ is said to be a 2-gap of $B$ and if $v_i,v_{i+4} \in B$ and $v_{i+1},v_{i+2},v_{i+3} \notin B$, then $\{v_{i+1},v_{i+2},v_{i+3}\}$ is said to be a 3-gap of $B$. Since $B$ is an adjacency basis of $C_n$, it has no gaps of size 4 or larger and it has at most one 3-gap. Moreover, every 2- or 3-gap must be neighboured by two 1-gaps and the number of gaps of either size is at most $\dim_A(C_n)$. We now differentiate the following cases for $C_n$:

\begin{enumerate}
\item $n=5k+1$, $k\ge 2$. In this case, by Lemma~\ref{LemmaAdimPathCycle}, $\dim_A(C_n)=2k$, and thus $n-\dim_A(C_n)=3k+1$. Since any 2-gap must be neighboured by two 1-gaps, any adjacency basis $B$ of $C_n$ has at most $k$ 2-gaps. Now, assume that $B$ has no 3-gaps. Then $|V(C_n)-B|=3k<3k+1=n-|B|$, which is a contradiction. Thus, any adjacency basis of $C_n$ has a 3-gap, \textit{i.e.} it is not a dominating set.
\item $n=5k+3$, $k\ge 1$. In this case, by Lemma~\ref{LemmaAdimPathCycle}, $\dim_A(C_n)=2k+1$, and thus $n-\dim_A(C_n)=3k+2$. As in the previous case, any adjacency basis $B$ of $C_n$ has at most $k$ 2-gaps. Now assume that $B$ has no 3-gaps. Then $|V(C_n)-B|=3k+1<3k+2=n-|B|$, which is a contradiction. Thus, any adjacency basis of $C_n$ has a 3-gap, \textit{i.e.} it is not a dominating set.
\end{enumerate}

By the set of cases above, the result holds for $C_n$.

Now, let $C'_n$ be the cycle obtained from $P_n$ by joining its leaves $v_1$ and $v_n$ by an edge. Let $V=V(P_n)=V(C'_n)$ and  let $B$ be an adjacency basis of $P_n$.  Since for two different vertices $x,y\in V$,   $d_{C'_n,2}(x,y)\ne d_{P_n,2}(x,y)$ if and only if $x,y\in \{ v_1,v_n\}$,  if $v_1,v_n \in B$ or $v_1,v_n \notin B$, then $B$ is an adjacency basis of $C_n$. Moreover, some vertex $w \in V-B$ satisfies $B \cap N_{P_n}(w)=B \cap N_{C'_n}(w)=\emptyset$, so $B$ is not a dominating set of $P_n$. We now treat the case where $v_1 \in B$ and $v_n \notin B$. If $v_{n-1} \notin B$ then $B$ is not a dominating set of $P_n$. If $v_{n-1} \in B$ and $v_2 \notin B$, we have that $d_{C'_n,2}(v_2,v_{n-1})=d_{P_n,2}(v_2,v_{n-1})=2 \neq 1=d_{P_n,2}(v_n,v_{n-1})=d_{C'_n,2}(v_n,v_{n-1})$, whereas for any other pair of different vertices $x,y \in V-B$ there exists $z \in B$ such that $d_{C'_n,2}(x,z)=d_{P_n,2}(x,z) \neq d_{P_n,2}(y,z)=d_{C'_n,2}(y,z)$, so $B$ is an adjacency basis of $C'_n$ where $\{v_n\}$ is a 1-gap. In consequence, some vertex $w \in V-(B\cup \{v_n\})$ satisfies $B \cap N_{P_n}(w)=B \cap N_{C'_n}(w)=\emptyset$, so $B$ is not a dominating set of $P_n$. Finally, if $v_2,v_{n-1} \in B$, then for any pair of different vertices $x,y \in V-B$ there exists $z \in B-\{v_1\}$ such that $d_{C'_n,2}(x,z)=d_{P_n,2}(x,z) \neq d_{P_n,2}(y,z)=d_{C'_n,2}(y,z)$, so $B$ is an adjacency basis of $C'_n$ where $\{v_n\}$ is a 1-gap. As in the previous case, some vertex $w \in V-(B\cup \{v_n\})$ satisfies $B \cap N_{P_n}(w)=B \cap N_{C'_n}(w)=\emptyset$, so $B$ is not a dominating set of $P_n$. The proof is complete.
\end{proof}

Lemma~\ref{lemmaNonExistAdjBDominatingCyclePath} allows us to give the following result.

\begin{proposition}\label{simAdjDimFamCoronaCase2GenPartCase1}
Let ${\cal G}$ be a family of connected non-trivial graphs on a common vertex set $V$. Let $P_n$ be a path graph of order $n \ge 7$, $n \equiv 1 \mod 5$ or $n \equiv 3 \mod 5$, and let $C_n$ be the cycle graph obtained from $P_n$ by joining its leaves by an edge. Let $B$ be a simultaneous adjacency basis of $\{P_n, C_n\}$. Then, for every family ${\cal H}={\cal H}_1 \cup {\cal H}_2$ such that ${\cal H}_1$ is composed by paths, ${\cal H}_1\subseteq {\cal G}_B(P_n)$, $P_n \in {\cal H}_1$, ${\cal H}_2$ is composed by cycles ${\cal H}_2\subseteq {\cal G}_B(C_n)$, and $C_n \in {\cal H}_2$, $$\Sd_A({\cal G} \odot {\cal H})=|V|\cdot\left(\left\lfloor\frac{2n+2}{5}\right\rfloor+1\right)-1.$$
\end{proposition}

\begin{proof}
Note that $B$ is an adjacency basis of both $P_n$ and $C_n$. Since $P_n \in {\cal H}_1$ and $C_n \in {\cal H}_2$, we have that $B$ is a simultaneous adjacency basis of ${\cal H}={\cal H}_1 \cup {\cal H}_2$ by Theorem~\ref{dimAPermFamily}. Moreover, since every $H \in {\cal H}_1$ is a path graph and every $H \in {\cal H}_2$ is a cycle we have that $\dim_A(H)=\Sd_A({\cal H})$ for every $H \in {\cal H}$, so every simultaneous adjacency basis of ${\cal H}$ is an adjacency basis of every $H \in {\cal H}$ and, by Lemma~\ref{lemmaNonExistAdjBDominatingCyclePath}, is not a dominating set of $H$. Thus, the result follows from Theorem~\ref{simAdjDimFamCoronaCase2}.
\end{proof}

It is worth noting that for a path graph $P_n$ and a cycle graph $C_n$, $n \ge 7$, $n \equiv 1 \mod 5$ or $n \equiv 3 \mod 5$, and an adjacency basis $B$ of both, the family ${\cal G}_B(P_n)$ contains $\left(n-\left\lfloor\frac{2n+2}{5}\right\rfloor\right)!$ path graphs, whereas the family ${\cal G}_B(C_n)$ contains $\left(n-\left\lfloor\frac{2n+2}{5}\right\rfloor\right)!$ cycle graphs.

\begin{proposition}\label{simAdjDimFamCoronaCase2GenPartCase2}
Let ${\cal G}$ be a family of connected non-trivial graphs on a common vertex set $V$ and let ${\cal H}=\{N_t \cup H_1,N_t \cup H_2,\ldots,N_t \cup H_k\}$, where $N_t$ is an empty graph of order $t \ge 1$ and $H_1,H_2,\ldots,H_k$ are connected non-trivial graphs on a common vertex set. Then, $$\Sd_A({\cal G} \odot {\cal H})=|V|\cdot\Sd_A({\cal H})+|V|-1.$$
\end{proposition}

\begin{proof}
Consider that the common vertex set of ${\cal H}$ has the form $V'=V(N_t) \cup V''$, where $V(N_t)$ and $V''$ are disjoint. Let $B$ be a simultaneous adjacency basis of ${\cal H}$, and let $B''=B \cap V''$. Consider an arbitrary graph $N_t \cup H \in {\cal H}$. The vertices of $N_t$ are false twins, so $V(N_t) \subseteq B$ if and only if there exists $v \in V''$ such that $B \cap N_{H}(v)=\emptyset$. If such $v$ exists, it is not dominated by $B$, so the result follows from Theorem~\ref{simAdjDimFamCoronaCase2}. Otherwise, $V(N_t)-B=\{v'\}$ and $B \cap N_{H}(v')=\emptyset$, so the result follows from Theorem~\ref{simAdjDimFamCoronaCase2}.
\end{proof}

As usual, given a graph $G$, we denote by $\gamma(G)$ the domination number of $G$.
 
\begin{theorem}{\rm \cite{RV-F-2013}}\label{TheOnlyPosibilitiesDimAdjCorona(b)}
Let $G$ be a  connected graph of order $n\ge 2$ and let $H$ be a non-trivial graph.  
If there exists an adjacency basis of $H$, which is also a dominating set and if, for any adjacency basis $S$ of $H$,  there exists $v\in V(H)-S$ such that $S\subseteq N_H(v)$, 
then  $$\dim_A(G\odot H)=n\cdot \dim_A(H)+\gamma (G).$$
\end{theorem}
 
 The \emph{simultaneous domination number} of a family ${\cal G}$, which we will denote as $\Sgamma({\cal G})$, is the minimum cardinality of a simultaneous dominating set. The next result is a generalisation of Theorem \ref{TheOnlyPosibilitiesDimAdjCorona(b)} to the case of ${\cal G} \odot {\cal H}$.

\begin{theorem}\label{simAdjDimFamCoronaCase3}
Let ${\cal G}$ be a family of connected non-trivial graphs on a common vertex set $V$ and let ${\cal H}$ be a family of non-trivial graphs on a common vertex set $V'$. If there exists a simultaneous adjacency basis of ${\cal H}$ which is also a simultaneous dominating set, and for every simultaneous adjacency basis $B$ of ${\cal H}$ there exist  $H \in {\cal H}$ and  $v \in V'-B$ such that  $B \subseteq N_H(v)$, then $$\Sd_A({\cal G} \odot {\cal H})=|V|\cdot\Sd_A({\cal H}) + \Sgamma({\cal G}).$$
\end{theorem}

\begin{proof}
We first address the proof of $\Sd_A({\cal G} \odot {\cal H}) \geq |V|\cdot\Sd_A({\cal H})+\Sgamma({\cal G})$. Let $U$ be a simultaneous adjacency basis of ${\cal G} \odot {\cal H}$, let $U_i=U \cap V'_i$, and let $U_0 = U \cap V$. By Remark~\ref{BiSimGenHi}, $U_i$ is a simultaneous adjacency generator for ${\cal H}_i$ for every $i \in V$. Consider the partition $\{V_1,V_2\}$ of $V$ defined as $$V_1=\{i \in V:\;|U_i|=\Sd_A({\cal H})\}\ \text{and}\ V_2=\{i \in V:\;|U_i| \geq \Sd_A({\cal H})+1\}.$$

For every $i \in V_1$, the set $U_i$ is a simultaneous adjacency basis of ${\cal H}_i$, so there exist $H \in {\cal H}$ and $x \in V'_i$ such that $U_i \subseteq N_H(x)$, causing $i$ and $x$ not to be distinguished by any $y \in U_i$ in any graph belonging to ${\cal G}\odot H$. Thus, either $i \in U_0$ or for every $G \in {\cal G}$ there exists $z \in U_0$ such that $d_{G \odot H,2}(i,z)=1 \neq 2=d_{G \odot H,2}(x,z)$. In consequence, $V_2 \cup U_0$ must be a simultaneous dominating set of ${\cal G}$, so $|V_2 \cup U_0| \geq \Sgamma({\cal G})$. Finally,

\begin{align*}
\Sd_A({\cal G} \odot {\cal H})&=\displaystyle{\sum_{i \in V_1}}|U_i|+\displaystyle{\sum_{i \in V_2}}|U_i|+|U_0|\\
&\geq\displaystyle{\sum_{i \in V_1}}\Sd_A({\cal H})+\displaystyle{\sum_{i \in V_2}}\left(\Sd_A({\cal H})+1\right)+|U_0|\\
&=|V|\cdot\Sd_A({\cal H})+|V_2|+|U_0|\\
&\geq|V|\cdot\Sd_A({\cal H})+|V_2 \cup U_0|\\
&\geq|V|\cdot\Sd_A({\cal H})+\Sgamma({\cal G}).
\end{align*}

Now, let $W$ be a simultaneous adjacency basis of ${\cal H}$ which is also a simultaneous dominating set of ${\cal H}$. Consider an arbitrary graph $G \odot H \in {\cal G}\odot{\cal H}$, and let $W_i=W \cap V'_i$. By analogy to the proof of Theorem~\ref{TheOnlyPosibilitiesDimAdjCorona(b)}, we have that $S=M\bigcup\left(\displaystyle{\bigcup_{i\in V}}W_i\right)$, where $M$ is a minimum simultaneous dominating set of ${\cal G}$, is an adjacency generator for $G \odot H$. Since $S$ does not depend on the election of $G$ and $H$, it is a simultaneous adjacency generator for ${\cal G} \odot {\cal H}$. Thus, $\Sd_A({\cal G} \odot {\cal H}) \leq \vert S \vert = |V|\cdot\Sd_A({\cal H})+\Sgamma({\cal G})$, so the equality holds.
\end{proof}

Several specific families for which the previous result holds will be described in Theorem~\ref{GodotJoinCase2} and Propositions~\ref{simAdjDimFamCoronaCase3FamGPartCase2} and~\ref{simAdjDimFamCoronaCase3FamGPartCase3}. Now, in order to present our next result, we need some additional definitions. Let $v \in V(G)$ be a vertex of a graph $G$ and let $G-v$ be the graph obtained by removing from $G$ the vertex $v$ and all its incident edges. Consider the following auxiliary domination parameter, which is defined in \cite{RV-F-2013}: $$\gamma'(G)=\underset{v \in V(G)}{\min}\{\gamma(G-v)\}$$

\begin{theorem}{\rm \cite{RV-F-2013}}\label{TheOnlyPosibilitiesDimAdjCorona(c)}
Let $H$ be a non-trivial graph such that some of its adjacency bases are also dominating sets, and some are not. If there exists an adjacency basis $S'$ of $H$ such that for every $v \in V(H)-S'$ it is satisfied that $S' \nsubseteq N_H(v)$, and for any adjacency basis $S$ of $H$ which is also a dominating set, there exists some $v \in V(H)-S$ such that $S \subseteq N_H(v)$, then for any connected non-trivial graph $G$, $$\dim_A(G \odot H)=n\cdot\dim_A(H)+\gamma'(G).$$
\end{theorem}

The following result is a generalisation of Theorem~\ref{TheOnlyPosibilitiesDimAdjCorona(c)} to the case of $G \odot {\cal H}$.

\begin{theorem}\label{simAdjDimFamCoronaCase4}
Let $G$ be a connected graph of order $n \geq 2$ and let ${\cal H}$ be a family of non-trivial graphs on a common vertex set $V'$ such that some of its simultaneous adjacency bases are also simultaneous dominating sets, and some are not. If there exists a simultaneous adjacency basis $B'$ of ${\cal H}$ such that $B' \nsubseteq N_H(v)$ for every $H \in {\cal H}$ and every $v \in V'-B'$, and for every simultaneous adjacency basis $B$ of ${\cal H}$ which is also a simultaneous dominating set there exist $H' \in {\cal H}$ and $w \in V'-B$ such that $B \subseteq N_{H'}(w)$, then $$\Sd_A(G \odot {\cal H})=n\cdot\Sd_A({\cal H})+\gamma'(G).$$
\end{theorem}

\begin{proof}
In the family $G\odot{\cal H}$, we have that $V=V(G)$. We first address the proof of $\Sd_A(G \odot {\cal H}) \geq n\cdot\Sd_A({\cal H})+\gamma'(G)$. Let $U$ be a simultaneous adjacency basis of $G \odot {\cal H}$, let $U_i=U \cap V'_i$, and let $U_0 = B \cap V$. By Remark~\ref{BiSimGenHi}, $U_i$ is a simultaneous adjacency generator for ${\cal H}_i$ for every $i \in V$. Consider the partition $\{V_1,V_2,V_3\}$ of $V$, where $V_1$ contains the vertices $i \in V$ such that $U_i$ is a simultaneous adjacency basis of ${\cal H}_i$ but is not a simultaneous dominating set, $V_2$ contains the vertices $i \in V$ such that $U_i$ is a simultaneous adjacency basis and a simultaneous dominating set of ${\cal H}_i$, and $V_3$ is composed by the vertices $i \in V$ such that $U_i$ is not a simultaneous adjacency basis of ${\cal H}_i$.

If $i,j \in V_1$, then there exist a graph $H \in {\cal H}$ and two vertices $v_i \in V'_i-U_i$ and $v_j \in V'_j-U_j$ such that $U_i \cap N_H(v_i)=\emptyset$ and $U_j \cap N_H(v_j)=\emptyset$. Thus, $i \in U_0$ or $j \in U_0$, so $|U_0 \cap V_1| \geq |V_1|-1$. If $i \in V_2$, then there exist $H \in {\cal H}$ and $x \in V'_i$ such that $U_i \subseteq N_H(x)$. In consequence, the pair $i,x$ is not distinguished by any $y \in U_i$, so either $i \in U_0$ or there exists $z \in U_0$ such that $d_{G \odot H,2}(i,z)=1 \neq 2=d_{G \odot H,2}(x,z)$. Therefore, at most one vertex of $G$ is not dominated by $U_0 \cup V_3$, so $|U_0 \cup V_3| \geq \gamma'(G)$. Finally,

\begin{align*}
\Sd_A(G \odot {\cal H})&=\displaystyle{\sum_{i \in V_1 \cup V_2}}|U_i| + \displaystyle{\sum_{i \in V_3}}|U_i| + |U_0|\\
&\geq \displaystyle{\sum_{i \in V_1 \cup V_2}}\Sd_A({\cal H}) + \displaystyle{\sum_{i \in V_3}}(\Sd_A({\cal H})+1) + |U_0|\\
&=n \cdot \Sd_A({\cal H})+|V_3|+|U_0|\\
&\geq n \cdot \Sd_A({\cal H})+|V_3 \cup U_0|\\
&\geq n \cdot \Sd_A({\cal H})+\gamma'(G).
\end{align*}

Now, let $W'$ be a simultaneous adjacency basis of ${\cal H}$ such that $W' \nsubseteq N_H(v)$ for every $H \in {\cal H}$ and every $v \in V-W'$, and assume that for any simultaneous adjacency basis $W$ of ${\cal H}$ which is also a simultaneous dominating set there exist $H' \in {\cal H}$ and $w \in V-W$ such that $W \subseteq N_{H'}(w)$. Let $W''$ be one of such simultaneous adjacency bases of ${\cal H}$. Consider an arbitrary graph $G \odot H \in G \odot{\cal H}$, let $W'_i=W' \cap V'_i$ and $W''_i=W'' \cap V'_i$. Additionally, let $M$ be a minimum dominating set of $G-n$, assuming without loss of generality that  $\gamma'(G)=\gamma(G-n)$, and let $S=M\bigcup W'_n\bigcup\left(\displaystyle{\bigcup_{i\in V-\{n\}}}W''_i\right)$. By analogy to the proof of Theorem~\ref{TheOnlyPosibilitiesDimAdjCorona(c)}, we have that $S$ is an adjacency generator for $G \odot H$. Since $S$ does not depend on the election of $G$ and $H$, it is a simultaneous adjacency generator for $G \odot {\cal H}$. Thus, $\Sd_A(G \odot {\cal H}) \leq \vert S \vert = n\cdot\Sd_A({\cal H})+\gamma'(G)$, so the equality holds.
\end{proof}

Consider the family $\{P_5,C_5\}$, where $C_5$ is obtained from $P_5$ by joining its leaves with an edge. Assume that $V(P_5)=V(C_5)=\{v_1,v_2,v_3,v_4,v_5\}$, $E(P_5)=\{v_1v_2, v_2v_3,$ $ v_3v_4, v_4v_5\}$ and $E(C_5)=E(P_5) \cup \{v_1v_5\}$. We have that the set $\{v_2,v_4\}$ is the sole simultaneous adjacency basis which is also a simultaneous dominating set and $v_3$ satisfies $\{v_2,v_4\} \subseteq N_{P_5}(v_3)$ and $\{v_2,v_4\} \subseteq N_{C_5}(v_3)$. Moreover, the set $\{v_1,v_5\}$ (as well as $\{v_2,v_3\}$ and $\{v_3,v_4\}$) is a simultaneous adjacency basis such that every vertex $v_x$ satisfies $N_{P_5}(v_x) \nsubseteq \{v_1,v_5\}$ and $N_{C_5}(v_x) \nsubseteq \{v_1,v_5\}$. These facts allow us to obtain examples where Theorem~\ref{simAdjDimFamCoronaCase4} applies. For instance, for any connected graph $G$ of order $n \ge 2$, we have that $\Sd_A(G \odot \{P_5,C_5\})=2n+\gamma'(G)$.

\subsection*{The case where the second factor is a family of join graphs}

Given two vertex-disjoint graphs $G=(V_{1},E_{1})$ and $H=(V_{2},E_{2})$, the \emph{join} of $G$ and $H$, denoted by $G+H$, is the graph with vertex set $V(G+H)=V_{1}\cup V_{2}$ and edge set $E(G+H)=E_{1}\cup E_{2}\cup \{uv\,:\,u\in V_{1},v\in V_{2}\}$. For two graph families ${\cal G}$ and ${\cal H}$, defined on common vertex sets $V_1$ and $V_2$, respectively, such that $V_1 \cap V_2 = \emptyset$, we define the family $${\cal G}+{\cal H}=\{G+H:\;G\in{\cal G},H\in{\cal H}\}.$$

In particular, if ${\cal G}=\{G\}$ we will use the notation $G+{\cal H}$. To begin our presentation, we introduce the following auxiliary result.

\begin{lemma}\label{everyBDominating}
Let ${\cal G}$ and ${\cal H}$ be two families of non-trivial graphs on common vertex sets $V_1$ and $V_2$, respectively. Then, every simultaneous adjacency basis of ${\cal G}+{\cal H}$ is a simultaneous dominating set of ${\cal G}+{\cal H}$.
\end{lemma}

\begin{proof}
Let $B$ be a simultaneous adjacency basis of ${\cal G}+{\cal H}$, let $W_1=B\cap V_1$ and $W_2=B\cap V_2$. Since no pair of different vertices $u,v \in V_2-W_2$ is distinguished in any $G+H\in{\cal G}+{\cal H}$ by any vertex from $W_1$, we have that $W_2$ is a simultaneous adjacency generator for ${\cal H}$ and, in consequence, $W_2 \neq \emptyset$. By an analogous reasoning we can see that $W_1$ is a simultaneous adjacency generator for ${\cal G}$ and, in consequence, $W_1 \neq \emptyset$. Moreover, every vertex in $V_1$ is dominated by every vertex in $W_2$, whereas every vertex in $V_2$ is dominated by every vertex in $W_1$, so $B$ is a dominating set for every $G+H \in {\cal G}+{\cal H}$.
\end{proof}

The following result, presented in \cite{Ramirez_Estrada_Rodriguez_2015}, characterizes a large number of families of the form ${\cal G}+{\cal H}$ whose simultaneous adjacency bases are formed by the union of simultaneous adjacency bases of ${\cal G}$ and ${\cal H}$.

\begin{theorem}{\rm \cite{Ramirez_Estrada_Rodriguez_2015}}\label{SimAdjDimJoinsFamilyCase1}
Let ${\cal G}$ and ${\cal H}$ be two families of non-trivial graphs on common vertex sets $V_1$ and $V_2$, respectively. If there exists a simultaneous adjacency basis $B$ of ${\cal G}$ such that for every  $G \in {\cal G}$ and every $v\in V_1$, $B\not\subseteq N_G(v)$, then $$\Sd_A({\cal G}+{\cal H})=\Sd_A({\cal G})+\Sd_A({\cal H}).$$
\end{theorem}

As discussed in the proof of Theorem~\ref{SimAdjDimJoinsFamilyCase1}, every simultaneous adjacency basis of a family ${\cal G}+{\cal H}$ satisfying the premises of the theorem is the union of a simultaneous adjacency basis of ${\cal H}$ and a simultaneous adjacency basis $B$ of ${\cal G}$ such that $B \nsubseteq N_G(v)$ for every $G \in {\cal G}$ and every $v \in V_1$.

\begin{theorem}\label{GodotJoinCase1}
Let ${\cal G}$ be a family of connected non-trivial graphs on a common vertex set $V$, and let ${\cal H}$ and ${\cal H}'$ be families of non-trivial graphs on common vertex sets $V'_1$ and $V'_2$, respectively. If there exist a simultaneous adjacency basis $B$ of ${\cal H}$ that satisfies $B \nsubseteq N_H(v)$ for every $H \in {\cal H}$ and every $v \in V'_1$, and a simultaneous adjacency basis $B'$ of ${\cal H}'$ that satisfies $B' \nsubseteq N_{H'}(v')$ for every $H' \in {\cal H}'$ and every $v' \in V'_2$, then $$\Sd_A({\cal G} \odot ({\cal H}+{\cal H}'))=|V| \cdot \Sd_A({\cal H})+|V| \cdot \Sd_A({\cal H}').$$
\end{theorem}

\begin{proof}
Let $B$ and $B'$ be simultaneous adjacency bases of ${\cal H}$ and ${\cal H}'$, respectively, that satisfy the premises of the theorem, and let $S=B\cup B'$. As shown in the proof of Theorem~\ref{SimAdjDimJoinsFamilyCase1}, $S$ is a simultaneous adjacency basis of ${\cal H}+{\cal H}'$. Moreover, since $B \nsubseteq N_H(v)$ for every $H \in {\cal H}$ and every $v \in V'_1$, and $B' \nsubseteq N_{H'}(v')$ for every $H' \in {\cal H}'$ and every $v' \in V'_2$, we have that $S \nsubseteq N_{H+H'}(x)$ for every $H+H' \in {\cal H}+{\cal H}'$ and every $x \in V'_1 \cup V'_2$. Finally, by Lemma~\ref{everyBDominating}, we have that $S$ is a simultaneous dominating set of ${\cal H}+{\cal H}'$, so $\Sd_A({\cal G} \odot ({\cal H}+{\cal H}'))=|V| \cdot \Sd_A({\cal H}+{\cal H}')=|V| \cdot \Sd_A({\cal H})+|V| \cdot \Sd_A({\cal H}')$ by Theorems~\ref{simAdjDimFamCoronaCase1} and~\ref{SimAdjDimJoinsFamilyCase1}.
\end{proof}

The following result is a direct consequence of Lemma~\ref{LemmaDiameter>6orPathorCycle} and Theorem~\ref{GodotJoinCase1}.

\begin{proposition}\label{simAdjDimFamCoronaCase1GenPartCase2}
Let ${\cal G}$ be a family of connected non-trivial graphs on a common vertex set $V$. Let ${\cal H}$ be a graph family on a common vertex set $V'_1$ of cardinality $|V'_1|\ge 7$ such that every $H \in {\cal H}$ is a path graph, a cycle graph, $D(H)\ge 6$, or $\mathtt{g}(H) \ge 5$ and $\delta(H) \ge 3$. Let ${\cal H}'$ be a graph family on a common vertex set $V'_2$ of cardinality $|V'_2|\ge 7$ satisfying the same conditions as ${\cal H}$.  Then, $$\Sd_A({\cal G} \odot ({\cal H}+{\cal H}'))=|V| \cdot \Sd_A({\cal H})+|V| \cdot \Sd_A({\cal H}').$$
\end{proposition}

In addition, following a reasoning analogous to that of the proofs of Propositions~\ref{simAdjDimFamCoronaCase1GenPartCase1} and~\ref{simAdjDimFamCoronaCase2GenPartCase1}, we obtain the following result as a consequence of Lemma~\ref{LemmaDiameter>6orPathorCycle} and Theorems~\ref{dimAPermFamily} and~\ref{GodotJoinCase1}.

\begin{proposition}\label{simAdjDimFamCoronaCase1GenPartCase3}
Let ${\cal G}$ be a family of connected non-trivial graphs on a common vertex set $V$. Let $H$ be a graph of order $n \ge 7$ which is a path graph, or a cycle graph, or satisfies $D(H)\ge 6$, or $\mathtt{g}(H) \ge 5$ and $\delta(H) \ge 3$. Let $H'$ be a graph of order $n' \ge 7$ that satisfies the same conditions as $H$. Let $B$ and $B'$ be adjacency bases of $H$ and $H'$, respectively. Then, for any pair of families ${\cal H} \subseteq {\cal G}_B(H)$ and ${\cal H}' \subseteq {\cal G}_{B'}(H')$ such that $H \in {\cal H}$ and $H' \in {\cal H}'$, $$\Sd_A({\cal G} \odot({\cal H}+{\cal H}'))=|V| \cdot \dim_A(H)+|V| \cdot \dim_A(H').$$
\end{proposition}

\begin{theorem}\label{GodotJoinCase2}
Let ${\cal G}$ be a family of connected non-trivial graphs on a common vertex set $V$, and let ${\cal H}$ and ${\cal H}'$ be families of non-trivial graphs on common vertex sets $V'_1$ and $V'_2$, respectively. If there exists a simultaneous adjacency basis $B$ of ${\cal H}$ that satisfies $B \nsubseteq N_H(v)$ for every $H \in {\cal H}$ and every $v \in V'_1$, and for every simultaneous adjacency basis $B'$ of ${\cal H}'$ there exist $H' \in {\cal H}$ and $v' \in V'_2$ such that $B' \subseteq N_{H'}(v')$, then $$\Sd_A({\cal G} \odot ({\cal H}+{\cal H}'))=|V| \cdot \Sd_A({\cal H})+|V| \cdot \Sd_A({\cal H}')+\Sgamma({\cal G}).$$
\end{theorem}

\begin{proof}
Let $S$ be a simultaneous adjacency basis of ${\cal H}+{\cal H}'$, let $W=S \cap V'_1$ and let $W'=S \cap V'_2$. As discussed in the proof of Theorem~\ref{SimAdjDimJoinsFamilyCase1}, $W$ and $W'$ are simultaneous adjacency bases of ${\cal H}$ and ${\cal H}'$, respectively. Since there exist $H' \in {\cal H}$ and $v' \in V'_2$ such that $W' \subseteq N_{H'}(v')$, we have that $S \subseteq N_{H+H'}(v')$ for any $H \in {\cal H}$ by the definition of the join operation. Moreover, by Lemma~\ref{everyBDominating}, $S$ is a simultaneous dominating set of ${\cal H}+{\cal H}'$, so $\Sd_A({\cal G} \odot ({\cal H}+{\cal H}'))=|V| \cdot \Sd_A({\cal H}+{\cal H}')+\Sgamma({\cal G})=|V| \cdot \Sd_A({\cal H})+|V| \cdot \Sd_A({\cal H}')+\Sgamma({\cal G})$ by Theorems~\ref{simAdjDimFamCoronaCase3} and~\ref{SimAdjDimJoinsFamilyCase1}.
\end{proof}

The following results are particular cases of Theorem~\ref{GodotJoinCase2}.

\begin{proposition}\label{simAdjDimFamCoronaCase3FamGPartCase2}
Let ${\cal G}$ be a family of connected non-trivial graphs on a common vertex set $V$. Let ${\cal H}$ be a graph family on a common vertex set $V'$ of cardinality $|V'|\ge 7$ such that every $H \in {\cal H}$ is a path graph, a cycle graph, $D(H)\ge 6$, or $\mathtt{g}(H) \ge 5$ and $\delta(H) \ge 3$. Let $K_t$ be a complete graph of order $t \ge 2$.  Then, $$\Sd_A({\cal G} \odot (K_t+{\cal H}))=|V| \cdot \Sd_A({\cal H})+|V|\cdot(t-1)+\Sgamma({\cal G}).$$
\end{proposition}

\begin{proof}
By Theorem~\ref{SimAdjDimJoinsFamilyCase1}, $\Sd_A(K_t+{\cal H})=\Sd_A({\cal H})+t-1$. Moreover, by Lemma~\ref{LemmaDiameter>6orPathorCycle}, every simultaneous adjacency basis $B$ of ${\cal H}$ satisfies $B \nsubseteq N_H(v)$ for every $H \in {\cal H}$ and every $v \in V'$. Furthermore, every adjacency basis of $K_t$ has the form $B'=V(K_t)-\{v\}$, where $v$ is an arbitrary vertex of $K_t$. Clearly, $B' \subseteq N_{K_t}(v)$, so the result follows from Theorem~\ref{GodotJoinCase2}.
\end{proof}

Following a reasoning analogous to that of the proofs of Propositions~\ref{simAdjDimFamCoronaCase1GenPartCase1} and~\ref{simAdjDimFamCoronaCase2GenPartCase1},
we obtain the following result as a consequence of Lemma~\ref{LemmaDiameter>6orPathorCycle} and Theorems~\ref{dimAPermFamily}, \ref{SimAdjDimJoinsFamilyCase1} and~\ref{GodotJoinCase2}. 

\begin{proposition}\label{simAdjDimFamCoronaCase3FamGPartCase3}
Let ${\cal G}$ be a family of connected non-trivial graphs on a common vertex set $V$. Let $H$ be a graph of order $n \ge 7$ which is a path graph, or a cycle graph, or satisfies $D(H)\ge 6$, or $\mathtt{g}(H) \ge 5$ and $\delta(H) \ge 3$. Let $K_t$ be a complete graph of order $t\ge 1$. Let $B$ be an adjacency basis of $H$. Then, for any family ${\cal H} \subseteq {\cal G}_B(H)$ such that $H \in {\cal H}$, $$\Sd_A({\cal G} \odot (K_t+{\cal H}))=|V| \cdot \dim_A(H)+|V|\cdot (t-1)+\Sgamma({\cal G}).$$
\end{proposition}

As an example of the previous result, consider an arbitrary family ${\cal G}$ composed by connected non-trivial graphs on a common vertex set $V$, a complete graph $K_t$ of order $t \ge 2$, a path graph $P_n$ of order $n \ge 7$, and the cycle graph $C_n$ obtained from $P_n$ by joining its leaves by an edge. For any simultaneous adjacency basis $B$ of $\{P_n,C_n\}$ and any family ${\cal H} \in {\cal G}_B(P_n) \cup {\cal G}_B(C_n)$ such that $P_n \in {\cal H}$ or $C_n \in {\cal H}$, we have that $$\Sd_A({\cal G} \odot (K_t+{\cal H}))=|V|\cdot\left(\left\lfloor\frac{2n+2}{5}\right\rfloor+t-1\right)+\Sgamma({\cal G}).$$

\end{document}